\documentclass{amsart}
\usepackage{amsmath,amssymb,amsfonts}
\usepackage[mathscr]{eucal}

\usepackage{enumerate}

\numberwithin{equation}{section}

\theoremstyle{plain}
\newtheorem{theorem}[equation]{Theorem}
\newtheorem{cor}[equation]{Corollary}
\newtheorem{prop}[equation]{Proposition}
\newtheorem{lemma}[equation]{Lemma}

\theoremstyle{definition}
\newtheorem{definition}[equation]{Definition}
\newtheorem{example}[equation]{Example}

\newtheorem{remark}[equation]{Remark}

\input xy
\xyoption{all}

\newcommand{\Deltaop}{{\bf \Delta}^{op}}

\newcommand{\SSets}{\mathcal{SS}ets}
\newcommand{\Top}{\mathcal Top}

\newcommand{\Map}{\text{Map}}
\newcommand{\map}{\text{map}}
\newcommand{\Hom}{\text{Hom}}

\newcommand{\nerve}{\text{nerve}}

\newcommand{\Sets}{\mathcal{S}ets}

\newcommand{\sesp}{\mathcal Se \mathcal Sp}

\newcommand{\css}{\mathcal{CSS}}

\newcommand{\ob}{\text{ob}}

\newcommand{\id}{\text{id}}

\newcommand{\Setop}{\mathcal Se \mathcal Top}
\newcommand{\Cstop}{\mathcal{CST}op}
\newcommand{\iso}{\operatorname{iso}}
\newcommand{\we}{\operatorname{we}}

\begin{document}

\title[Equivariant complete Segal spaces]{Equivariant complete Segal spaces}

\author[J.E.\ Bergner]{Julia E.\ Bergner}

\email{bergnerj@member.ams.org}

\author[S.G.\ Chadwick]{Steven Greg Chadwick}

\email{chadwick@math.ucr.edu}

\address{Department of Mathematics, University of California, Riverside, CA 92521}

\date{\today}



\thanks{The first-named author was partially supported by NSF grant DMS-1105766 and CAREER award DMS-1352298.}

\begin{abstract}
In this paper we give a model for equivariant $(\infty, 1)$-categories.  We modify an approach of Shimakawa for equivariant $\Gamma$-spaces to the setting of simplicial spaces.  We then adapt Rezk's Segal and completeness conditions to fit with this setting.  
\end{abstract}

\maketitle

\section{Introduction}

Two areas of much recent interest in homotopy theory have been that of homotopical categories, or $(\infty, 1)$-categories, and that of equivariant homotopy theory.  In this paper, we investigate an approach to merging these two areas by considering equivariant homotopical categories.  Specifically, we consider the complete Segal spaces of Rezk \cite{rezk} and incorporate actions by discrete or topological groups.

Our method for modeling these objects is inspired by Shimakawa's model for equivariant $\Gamma$-spaces \cite{shim}.  Originating with Segal's work in \cite{segal}, many constructions using the category $\Gamma$ of finite sets have analogues for the category $\Delta$ of finite ordered sets.  Segal considered $\Gamma$-spaces, or contravariant functors from $\Gamma$ to the category of spaces satisfying a so-called Segal condition; the space associated to a singleton set then inherits the structure of an infinite loop space.  Imposing a similar condition for functors from $\Delta$ instead, we obtain topological monoids, or, with some modification, the Segal spaces of Rezk \cite{rezk}.

Given a group $G$, Shimakawa considers $\Gamma$-$G$-spaces, or contravariant functors from $\Gamma$ to the category of spaces equipped with a $G$-action.  Here, we replace $\Gamma$ with $\Delta$ to obtain Segal $G$-spaces.  Applying Rezk's completeness condition, we get a model for $G$-equivariant complete Segal spaces.  Alternatively, we could ask that the degree zero space of a Segal $G$-space be discrete, leading to a theory of $G$-equivariant Segal categories.  Because we work in the setting of model categories, these results can also be considered to be a generalization of the work of Santhanam \cite{san}.


A more abstract approach to this problem is taken by the first-named author in \cite{ginfty1}, where most of the known models for $(\infty, 1)$-categories are shown to satisfy an axiomatization for when a model category has an associated model category of $G$-objects where weak equivalences are defined via fixed-point functors \cite{bmoopy}, \cite{stephan}.  There we consider the case of the action of discrete groups $G$ for any model for $(\infty, 1)$-categories, and for the action of simplicial groups $G$ on the models which have the additional structure of a simplicial model category, namely Segal categories and complete Segal spaces.  Here, we regard complete Segal spaces topologically rather than simplicially, and so extend to a case where we have the structure of a topological model category and hence can consider actions by compact Lie groups.

In Section \ref{back}, we give a brief review of the homotopy theory of $G$-spaces and of simplicial methods.  In Section \ref{segalG}, we introduce (complete) Segal $G$-spaces.  Then, in Section \ref{classify}, we show that an equivariant version of Rezk's classifying diagram produces examples from $G$-categories.  Finally, in Section \ref{functor}, we connect our approach here with that of Stephan for general topological model categories.

\section{Background} \label{back}

\subsection{The model structure for $G$-spaces}

Let $\Top$ denote the category of compactly generated Hausdorff topological spaces.  Quillen proved in \cite{quillen} that $\Top$ admits a model structure in which a map $f \colon X \rightarrow Y$ of topological spaces is a weak equivalence if $f$ induces isomorphisms $f_* \colon \pi_i(X) \rightarrow \pi_i(Y)$ for all $i \geq 0$, a fibration if it is a Serre fibration, and a cofibration if it has the left lifting property with respect to the acyclic fibrations \cite[7.10.6]{hirsch}.  This model structure is additionally cofibrantly generated, with the sets of inclusion maps
\[ I = \{S^{n-1} \rightarrow D^n \mid n \geq 0 \} \]
and
\[ J= \{i_0 \colon D^n \rightarrow D^n \times I \mid n \geq 0\} \]
as generating cofibrations and generating acyclic cofibrations, respectively.

Given a group $G$, let $G\Top$ denote the category of $G$-spaces and $G$-maps.  Given a $G$-space $X$ and a subgroup $H$ of $G$, define the fixed-point subspace
\[ X^H = \{x \in X \mid h \cdotp x= x \text{ for all } h \in H\} \subseteq X. \]
The product $G/H \times X$ is a $G$-space with the diagonal action $\gamma(gH, x) = (\gamma gH, \gamma x)$ for all $\gamma \in G$.

Let $X$ and $Y$ be $G$-spaces.  Regarding $X$ and $Y$ as objects of $\Top$, we have the mapping space $\Map_{\Top}(X,Y)$.  However, the $G$-actions on $X$ and $Y$ induce a $G$-action on $\Map_{\Top}(X,Y)$ by conjugation; given a map $f \colon X \rightarrow Y$, define $g \cdotp f$ by
\[ (g \cdotp f)(x) = g \cdotp f(g^{-1} \cdotp x). \]
The space of $G$-maps $\Map_{G\Top}(X,Y)$ is then defined to be $\Map_{\Top}(X,Y)^G$, so that $G\Top$ is enriched over $\Top$.

While the mapping spaces in $G\Top$ admit $G$-actions, the category $G\Top$ is not enriched in itself.  Rather, it is enriched in $\Top_G$, the category whose objects are the $G$-spaces and whose morphisms are all continuous maps. (Observe that also the category $\Top_G$ is enriched in $G\Top$.)  More details about enriched equivariant categories can be found in \cite[II.1]{mmss}.  However, in this paper we only use the fact that $G\Top$ is topological.

\begin{definition}
A category $\mathcal C$ is \emph{topological} if it is enriched in the category $\Top$, i.e., if for any objects $X$ and $Y$ of $\mathcal C$, there is a space $\Map_\mathcal C(X,Y)$ together with an associative, continuous composition.
\end{definition}

We recall the following definition.

\begin{definition}
A model category $\mathcal M$ is a \emph{topological model category} if it is a topological category satisfying the following conditions.
\begin{enumerate}
\item The category $\mathcal M$ is tensored and cotensored over $\Top$, so that, given any objects $X$ and $Y$ of $\mathcal M$ and topological space $A$, there is an object $X \otimes A$ of $\mathcal M$ and a topological space $Y^A$ such that there are natural homeomorphisms
\[ \Map_\mathcal M(X \otimes A, Y) \cong \Map_{\Top}(A, \Map_{\mathcal M}(X,Y)) \cong \Map_{\mathcal M}(X, Y^A). \]

\item If $i \colon A \rightarrow B$ is a cofibration and $p \colon X \rightarrow Y$ is a fibration in $\mathcal M$, then the induced map of topological spaces
\[ \Map_{\mathcal M}(i^*, p_*) \colon \Map_{\mathcal M}(B, X) \rightarrow \Map_\mathcal M(A,X) \times_{\Map_\mathcal M(A,Y)} \Map_\mathcal M(B,Y) \]
is a fibration which is a weak equivalence if either $i$ or $p$ is.
\end{enumerate}
\end{definition}

When we consider topological model categories, we use the notation $\Map^h(X,Y)$ to denote the homotopy invariant mapping space, given by taking the mapping space $\Map(X^c, Y^f)$, where $X^c$ is a cofibrant replacement for $X$ and $Y^f$ is a fibrant replacement for $Y$.

\begin{theorem} \cite[9.3.7]{hirsch}
The following conditions are equivalent for a model category $\mathcal M$ which is tensored and cotensored over $\Top$.
\begin{enumerate}
\item The model category $\mathcal M$ is a topological model category.

\item (Pushout-product axiom) If $i \colon A \rightarrow B$ is a cofibration in $\mathcal M$ and $j \colon C \rightarrow D$ is a cofibration in $\Top$, then the induced map
\[ A \times D \cup_{A \otimes C} B \otimes C \rightarrow B \otimes D \]
is a cofibration in $\mathcal M$ which is an acyclic cofibration if either $i$ or $j$ is.
\end{enumerate}
\end{theorem}

\begin{remark}
Observe that we are using exponential notation for two purposes here, both for the fixed points of an action and for cotensoring with a space.  We hope that the usage is clear from the context; typically groups are denoted here by $G$ or $H$.
\end{remark}

We have the following relationship between $\Top$ and its equivariant analogue $G\Top$.

\begin{lemma} \label{topGtop}
For any subgroup $H$ of $G$, there is an adjunction
\[ G/H \times (-) \colon \Top \rightleftarrows G\Top \colon (-)^H. \]
Specifically, for a $G$-space $X$ and a  space $A$ with trivial $G$-action, there is a homeomorphism
\[ \Map_{G\Top}(G/H \times A, X) \cong \Map_{\Top}(A, X^H). \]
\end{lemma}

\begin{proof}
Given a $G$-map $f \colon G/H \times A \rightarrow X$, define the map $f_\sharp \colon A \rightarrow X^H$ by $f_\sharp(a) = f(H, a)$.  Alternatively, given a map $f \colon A \rightarrow X^H$, define a $G$-map $f^\sharp \colon G/H \times A \rightarrow X$ by $f^\sharp(gH, a) = g \cdotp f(a)$.    One can check that these two constructions are inverse to one another, giving the desired isomorphism of sets.
\end{proof}

This adjunction can be used to define sets of generating cofibrations and generating acyclic cofibrations for $G\Top$, giving it the structure of a cofibrantly generated model category.  Namely, these sets are
\[ I_G = \{ G/H \times S^{n-1} \rightarrow G/H \times D^n \mid n \geq 0, H \leq G\} \]
and
\[ J_G = \{G/H \times D^n \rightarrow G/H \times D^n \times I \mid n \geq 0, H \leq G\}, \]
respectively.

\begin{theorem} \cite[III.1.8]{mm}
The category $G\Top$ admits the structure of a cofibrantly generated, cellular, proper topological model category, where a $G$-map $f \colon X \rightarrow Y$ is a weak equivalence or fibration if the induced map $f^H \colon X^H \rightarrow Y^H$ is a weak equivalence or fibration in $\Top$ for every subgroup $H$ of $G$.   The sets $I_G$ and $J_G$ defined above are sets of generating cofibrations and generating acyclic cofibrations, respectively.
\end{theorem}

\subsection{Simplicial objects}

Recall that $\Delta$ is the category whose objects are finite ordered sets $[n]=\{1 \leq \cdots \leq n\}$ and whose morphisms are weakly order-preserving functions.  A \emph{simplicial set} is a functor $\Deltaop \rightarrow \Sets$, where $\Sets$ denotes the category of sets.  More generally, given a category $\mathcal C$, a \emph{simplicial object in} $\mathcal C$ is a functor $\Deltaop \rightarrow \mathcal C$.  In this paper we are  mostly interested in the case where $\mathcal C$ is the category of $G$-spaces described above.  We denote the category of simplicial objects in $\mathcal C$ by $\mathcal C^{\Deltaop}$.

If $\mathcal C$ has the additional structure of a model category, then one can consider model structures on $\mathcal C^{\Deltaop}$.  One can always take the \emph{projective model structure}, in which the weak equivalences and fibrations are taken to be levelwise weak equivalences and fibrations, respectively, in $\mathcal C$ \cite[11.6.1]{hirsch}.

Because $\Deltaop$ has the additional structure of a Reedy category \cite[15.1.2]{hirsch}, we can also consider the Reedy model structure on $\mathcal C^{\Deltaop}$ \cite[15.3.4]{hirsch}, which again has weak equivalences defined levelwise.

\section{Segal $G$-spaces} \label{segalG}

In this section, we consider certain kinds of simplicial objects in the category of $G$-spaces.  We begin with some terminology.

\begin{definition}
Let $G$ be a group.  A \emph{simplicial} $G$-\emph{space} is a functor $\Deltaop \rightarrow G\Top$.
\end{definition}

Let $G\Top^{\Deltaop}$ denote the category of simplicial $G$-spaces.  This category is enriched in topological spaces; we can describe the enrichment as follows.  We topologize the set $\Hom_{G\Top^{\Deltaop}}(X,Y)$ as a subspace of the product space
\[ \prod_{[n]} \Map_{G\Top}(X_n, Y_n). \]
Furthermore, this category is tensored and cotensored in $\Top$; given a simplicial $G$-space $X$ and topological space $A$, define $X \otimes A$ by $(X \otimes A)_n = X_n \times A$ with the diagonal action, and define $X^A$ by $(X^A)_n = \Map_{\Top}(A, X_n)$ with $G$-action given by conjugation.  Using these definitions, one can check that there are homeomorphisms
\[ \Map_{G\Top^{\Deltaop}}(X \otimes A, Y) \cong \Map_{\Top}(A, \Map_{G\Top^{\Deltaop}}(X,Y)) \cong \Map_{G\Top^{\Deltaop}}(X, Y^A). \]

Consider $G\Top^{\Deltaop}$ with the projective model structure.  We can extend Lemma \ref{topGtop} to the following result.

\begin{theorem}
The adjunction
\[ G/H \times - \colon \Top \rightleftarrows G\Top \colon (-)^H \]
lifts to an adjunction
\[ \Top^{\Deltaop} \rightleftarrows G\Top^{\Deltaop} \]
where, given any simplicial $G$-space $X$ and any simplicial space $A$, there is a homeomorphism
\[ \Map_{G\Top^{\Deltaop}}(G/H \times A, X) \cong \Map_{\Top^{\Deltaop}}(A, X^H). \]
Here, $X^H \colon \Deltaop \rightarrow \Top$ denotes the functor defined by $(X^H)_n = (X_n)^H$, where on the right-hand side we take the usual $H$-fixed points of the space $X_n$.  Furthermore, this adjunction defines a Quillen pair between projective model structures.
\end{theorem}

Now, we want to consider simplicial $G$-spaces which behave like categories (with a $G$-action) up to homotopy.  More specifically, we would like to show that the Segal spaces of Rezk \cite{rezk}, which are simplicial objects in the category of simplicial sets, can be defined in the setting of simplicial $G$-spaces.  Thus, our first goal is to understand Segal maps in simplicial $G$-spaces.

Recall that in $\Delta$ there are maps $\alpha^i \colon [1] \rightarrow [n]$ for $0 \leq i \leq n-1$ defined by $0 \mapsto i$ and $1 \mapsto i+1$ which in turn give maps $\alpha_i \colon [n] \rightarrow [1]$ in $\Deltaop$.  Given a simplicial $G$-space $X$ we get induced maps $X(\alpha_i) \colon X_n \rightarrow X_1$.

\begin{definition}
Let $X$ be a simplicial $G$-space.  For any $n \geq 2$, the \emph{Segal map} is
\[ \varphi_n \colon X_n \rightarrow \underbrace{X_1 \times_{X_0} \cdots \times_{X_0} X_1}_n \]
induced by all the maps $X(\alpha_i)$.
\end{definition}

More precisely, for each $n \geq 0$ define the simplicial set
\[ G(n) = \bigcup_{i=0}^{n-1} \alpha^i \Delta[1] \subseteq \Delta[1] \]
equipped with a trivial $G$-action.  Regard $G(n)$ as a discrete simplicial space.
Then the Segal map $\varphi_n$ can be obtained as
\[ \Map^h_{G\Top^{\Deltaop} }(\Delta[n], X) \rightarrow \Map^h_{G\Top^{\Deltaop}}(G(n), X). \]

\begin{definition}
A Reedy fibrant simplicial $G$-space is a \emph{Segal} $G$-\emph{space} if the Segal maps are $G$-equivalences for all $n \geq 2$.
\end{definition}

Recall that, given a model category, one can localize with respect to a set of maps to obtain a new model structure on the same category, such that all maps in the given set become weak equivalences.  The cofibrations stay the same, but the fibrations change accordingly.  The new fibrant objects are called \emph{local objects} with respect to the localized model structure, and a map is a weak equivalence precisely if it is a \emph{local equivalence}, or map $A \rightarrow B$ such that the map
\[ \Map^h(B, W) \rightarrow \Map^h(A, W) \]
is a weak equivalence of simplicial sets for any local object $W$ \cite[3.1.1]{hirsch}.
In the case of (non-equivariant) simplicial spaces, one can obtain a model structure by localizing the Reedy model structure with respect to the set
\[ S=\{G(n) \rightarrow \Delta[n] \mid n \geq 2\} \]
so that the fibrant objects are precisely the Segal spaces \cite{rezk}.

Applying the adjunction between simplicial spaces and simplicial $G$-spaces, the Segal $G$-spaces should be local objects with respect to the set of maps
\[ S_G=\{G/H \times G(n) \rightarrow G/H \times \Delta[n] \mid n\geq 2, H \leq G\}. \]

\begin{theorem}
There is a topological model structure $G\Setop$ on the category of functors $\Deltaop \rightarrow G\Top$ whose fibrant objects are precisely the Segal $G$-spaces.
\end{theorem}

\begin{proof}
We localize the Reedy model structure with respect to the set $S_G$; the existence of the localized model structure follows from \cite[4.1.1]{hirsch}.  It remains to check that the local objects are the Segal $G$-spaces, and that this model structure is topological.

To check that the local objects are as described, first observe that, since the functor $(-)^H$ is a right adjoint and hence preserves pullbacks, we know that
\[ (X_1 \times_{X_0} \cdots \times_{X_0} X_1)^H = (X^H)_1 \times_{(X_0)^H} \cdots \times_{(X_0)^H} (X_1)^H. \]
Then we have the following chain of weak equivalences:
\[ \begin{aligned}
X_n^H & \cong \Map^h_{\Top}(\Delta^0, X_n^H) \\
& \cong \Map^h_{\Top^{\Deltaop}}(\Delta[n], X^H) \\
& \cong \Map^h_{G\Top^{\Deltaop}}(G/H \times \Delta[n], X) \\
& \simeq \Map^h_{G\Top^{\Deltaop}}(G/H \times G(n), X) \\
& \cong \Map^h(G/H \times \Delta[1], X)\times_{\Map^h(G/H \times \Delta[0], X)} \cdots \times_{\Map^h(G/H \times \Delta[0], X)} \Map^h(G/H \times \Delta[1],X) \\
& \cong \Map^h_{\Top^{\Deltaop}}(\Delta[1], X^H) \times_{\Map^h(\Delta[0], X^H)} \cdots \times_{\Map^h(\Delta[0], X^H)} \Map^h_{\Top^{\Deltaop}}(\Delta[1], X^H) \\
& \cong \Map^h_{\Top^{\Deltaop}}(\Delta[0] , X^H_1) \times_{\Map^h(\Delta[0], X^H_0)} \cdots \times_{\Map^h(\Delta[0], X^H_0)} \Map^h_{\Top^{\Deltaop}}(\Delta[0], X^H_1) \\
& \cong X_1^H \times_{X_0^H} \cdots \times_{X_0^H} X_1^H.
\end{aligned} \]
It follows that the local objects are precisely the Segal $G$-spaces.

We now turn to proving that $G\Setop$ is a topological model category.  We know that the underlying category $G\Top^{\Deltaop}$ is topological and both tensored and cotensored over $\Top$.  Therefore, we need only verify the pushout-product axiom.  Suppose that $i \colon A \rightarrow B$ is a cofibration in $G\Setop$ and $j \colon C \rightarrow D$ is a cofibration in $\Top$.  Since the Reedy model structure is topological, and cofibrations do not change under localizations, we know that the pushout-product map
\[ A \otimes D \cup_{A \otimes C} B \otimes C \rightarrow B \otimes D \]
is a cofibration in $G\Setop$.  It only remains to check that if $i$ or $j$ is a weak equivalence, then so is the pushout-product map.

Let $W$ be an $S_G$-local object, i.e., a Segal $G$-space.  Then the induced map of spaces
\[ \Map^h(B \otimes D, W) \rightarrow \Map^h(A \otimes D \cup_{A \otimes C} B\otimes C, W) \]
is a weak equivalence if and only if the diagram
\[ \xymatrix{\Map^h(B \otimes D, W) \ar[r] \ar[d] & \Map^h(B \otimes C, W) \ar[d] \\
\Map^h(A \otimes D, W) \ar[r] & \Map^h(A \otimes C, W)} \]
is a homotopy pullback square.  However, this square is equivalent to the square
\begin{equation} \label{pullback}
\xymatrix{\Map^h(B, W^D) \ar[r] \ar[d] & \Map^h(B, W^C) \ar[d] \\
\Map^h(A, W^D) \ar[r] & \Map^h(A, W^C).}
\end{equation}

Let us verify that $W^C$ and $W^D$ are still $S_G$-local.  Given any subgroup $H$ of $G$, consider the diagram
\[ \xymatrix{(W^C)^H_n \ar[r] \ar[d]_\simeq & (W^C)^H_1 \times_{(W^C)^H_0} \cdots \times_{(W^C)^H_0} (W^C)^H_1 \ar[d]^\simeq \\
\Map^h_{\Top}(C, W^H_n) \ar[r]^-\simeq & \Map^h(C, W^H_1) \times_{\Map^h(C, W^H_0)} \cdots \times _{\Map^h(C, W^H_0)} \Map^h(C, W^H_1).} \]
Observe that the object in the bottom right-hand corner is weakly equivalent to
\[ \Map^h(C, W^H_1 \times_{W^H_0} \cdots \times_{W^H_0} W^H_1). \]
The bottom horizontal arrow is a weak equivalence since $W$ is an $S_G$-local object, so it follows that the top horizontal map must also be a weak equivalence.  Therefore, $W^C$ is $S_G$-local.

Now suppose that $i \colon A \rightarrow B$ is an $S_G$-local equivalence.  Then, since $W^C$ and $W^D$ are $S_G$-local, the vertical arrows in \eqref{pullback} are weak equivalences; since they are also fibrations, we get that the diagram \eqref{pullback} is a (homotopy) pullback square.

Lastly, suppose instead that $j \colon C \rightarrow D$ is a fibration in $\Top$.  Then observe that there are weak equivalences
\[ (W^C)^H_n \simeq \Map^h_{\Top}(C, W^H_n) \simeq \Map^h_{\Top}(D, W^H_n) \simeq (W^D)^H_n. \]
Therefore, the map $W^D \rightarrow W^C$ is a levelwise weak equivalence.  It follows that the horizontal maps in \eqref{pullback} are acyclic fibrations and therefore the diagram is a (homotopy) pullback square.
\end{proof}

We now turn to complete Segal $G$-spaces.   Let $E$ be the nerve of the category with two objects and a single isomorphism between them, regarded as a discrete simplicial space.  Complete Segal spaces are those Segal spaces which are additionally local with respect to the inclusion $\Delta[0]\rightarrow E$.   In the equivariant setting, we want to localize with respect to the set
\[ T= \{G/H \times E \rightarrow G/H \times \Delta[0] \mid H \leq G\}. \]

\begin{theorem}
There is a topological model structure $G\Cstop$ on the category of simplicial $G$-spaces in which the fibrant objects are precisely the complete Segal $G$-spaces.
\end{theorem}

\begin{proof}
We localize the model category $G\sesp$ with respect to the set $T$.  Establishing that the model structure is topological can be done similarly to the proof for $G\sesp$.
\end{proof}

We conclude this section by applying some of the language of simplicial $G$-categories to Segal $G$-spaces.

For a Segal $G$-space $W$, define the \emph{objects} of $W$ to be the set of points in $W$, denoted by $\ob(W)$.  Given $x,y \in \ob(W)$, define the mapping $G$-space between them to be the pullback
\[ \xymatrix{\map_W^G(x,y) \ar[r] \ar[d] & W_1 \ar[d] \\
\{(x,y)\} \ar[r] & W_0 \times W_0. } \]
Since $W$ is assumed to be Reedy fibrant, the right-hand vertical map is a fibration, so that the mapping space is actually a homotopy pullback in the category of $G$-spaces.  Given an object $x$, define its identity map to be $\id_x= s_0(x) \in \map_W^G(x,x)_0$.

Just as in the case of ordinary Segal spaces \cite{rezk}, one can define (non-unique) composition of mapping $G$-spaces, homotopy equivalences, and the homotopy $G$-category of a Segal $G$-space.

%
%

\section{Complete Segal $G$-spaces from $G$-categories} \label{classify}

In this section we generalize the classifying diagram and classification diagram constructions of Rezk \cite{rezk} to the $G$-equivariant setting.

Let $\mathcal C$ be a small category equipped with an action of $G$, which can be thought of as a functor $G \rightarrow \mathcal Cat$.  Then $\nerve(\mathcal C)$ is a simplicial $G$-set; the action is given by $(g \cdotp F)(i) = g \cdotp F(i)$ for any $g \in G$, $F \colon [n] \rightarrow \mathcal C$, and $0 \leq i \leq n$.  Taking the classifying space
\[ B(\mathcal C) = |\nerve(\mathcal C)| \]
gives a $G$-space.

Let $\mathcal C^{[n]}$ denote the category whose objects are functors $[n] \rightarrow \mathcal C$ and whose morphisms are natural transformations.  Let $\iso(\mathcal C)^{[n]}$ denote the maximal subgroupoid of $\mathcal C^{[n]}$.

\begin{definition}
Given a small $G$-category $\mathcal C$, define its $G$-\emph{classifying diagram} $N^G(\mathcal C)$ to be the simplicial $G$-space defined by
\[ N^G(\mathcal C)_n = B(\iso(\mathcal C)^{[n]}). \]
\end{definition}

\begin{prop}
For any small $G$-category $\mathcal C$, the $G$-classifying diagram $N^G(\mathcal C)$ is a complete Segal $G$-space.
\end{prop}

\begin{proof}
We know that the $G$-classifying diagram is a complete Segal space, using Rezk \cite{rezk}.  Since it is defined by classifying spaces at each level, which inherit a $G$-action, we get a complete Segal $G$-space.
\end{proof}

\begin{example}
Let $\mathcal C$ be a small category and $G\mathcal C$ its category of $G$-objects, or functors $G \rightarrow \mathcal C$, for some discrete group $G$. Then we can regard $G\mathcal C$ as a $G$-category, or functor $G \rightarrow \mathcal Cat$, as follows.  On the level of objects, we can think of $G \mathcal C$ as defining a functor which takes the single object of $G$ to the category $G\mathcal C$.  Given any morphism $g \in G$, it defines an automorphism of $G \mathcal C$ which is the identity on any object $F \colon G \rightarrow \mathcal C$ (taking the object of $G$ to some object $C$ of $\mathcal C$) but sends a morphism defined by $(h \mapsto (h \colon C \rightarrow C))$ to the morphism $(h \mapsto (hg \colon C \rightarrow C))$.

Thus, we can take the classifying diagram $N^G(G\mathcal C)$ and obtain a complete Segal $G$-space.
\end{example}

More generally, let $\mathcal M$ be a model category or a category with weak equivalences equipped with a $G$-action.  Let $\we (\mathcal M)^{[n]}$ be the category whose objects are the functors $[n] \rightarrow \mathcal M$ and whose morphisms are natural transformations given by levelwise weak equivalences in $\mathcal M$.

\begin{definition}
Given a model category or category with weak equivalences $\mathcal M$ equipped with a $G$-action, define its $G$-\emph{classification diagram} to be the simplicial $G$-space $N^G(\mathcal M)$ defined by
\[ N^G(\mathcal M)_n = B(\we(\mathcal M)^{[n]}). \]
\end{definition}

\begin{prop}
Let $\mathcal M$ be a model category or category with weak equivalences equipped with a $G$-action.  Then a Reedy fibrant replacement of its $G$-classification diagram is a complete Segal $G$-space.
\end{prop}

\begin{proof}
This result was proved by Rezk \cite[8.3]{rezk} for the non-equivariant case.  His argument still holds in the setting of $G$-spaces.
\end{proof}

\section{Connection to the functor approach} \label{functor}

In this section, we compare the approach that we have taken in this paper to the approach of Stephan, which regards $G$-objects in a category $\mathcal C$ and functors $G \rightarrow \mathcal C$.  We begin with the statement of his general result.

\begin{theorem} \cite{stephan}, \cite{bmoopy} \label{cellular}
Let $G$ be a group, and let $\mathcal C$ be a cofibrantly generated model category.  Suppose that, for each subgroup $H$ of $G$, the fixed point functor $(-)^H$ satisfies the following cellularity conditions:
\begin{enumerate}
\item \label{cell1} the functor $(-)^H$ preserves filtered colimits of diagrams in $\mathcal C^G$,

\item \label{cell2} the functor $(-)^H$ preserves pushouts of diagrams where one arrow is of the form
\[ G/K \otimes f \colon G/K \otimes A \rightarrow G/K \otimes B \]
for some subgroup $K$ of $G$ and $f$ a generating cofibration of $\mathcal C$, and

\item \label{cell3} for any subgroup $K$ of $G$ and object $A$ of $\mathcal C$, the induced map
\[ (G/H)^K \otimes A \rightarrow (G/H \otimes A)^K \]
is an isomorphism in $\mathcal C$.
\end{enumerate}
Then the category $\mathcal C^G$ admits the $G$-model structure.
\end{theorem}

We have been working in the case that $\mathcal C = \css$, the model structure for complete Segal spaces, but taken as simplicial objects in $\Top$ rather than in $\SSets$.  It is immediate from results of Stephan, but stated in \cite{ginfty1}, that the three conditions are satisfied in the simplicial setting.  In the topological setting, we can use the fact that the category $\Top$ satisfies the conditions and apply levelwise.  The result is a category equivalent to the one we have defined.

However, Stephan also includes the following result.

\begin{theorem} \cite{stephan}
Let $G$ be a compact Lie group, and let $\mathcal C$ be a cofibrantly generated topological model category.  Suppose that, for each subgroup $H$ of $G$, the fixed point functor $(-)^H$ satisfies the cellularity conditions of Theorem \ref{cellular}.  Then the category $\mathcal C^G$ admits the $G$-model structure and is a topological model category.
\end{theorem}

\begin{cor}
For any compact Lie group $G$, the category of simplicial $G$-spaces admits the structure of a topological model category in which the fibrant objects are the complete Segal $G$-spaces.
\end{cor}

\end{document}